\documentclass[11pt]{amsart}
\usepackage{amsfonts,amssymb,graphicx}
\usepackage{epsfig,amsmath,latexsym}
\usepackage{amscd, amssymb, amsthm, amsmath,eucal, verbatim}
\usepackage{epstopdf}
\usepackage{epstopdf}
\newtheorem{theorem}{Theorem}[section]
\newtheorem{lemma}{Lemma}[section]

 \newcommand{\bproof}{\paragraph {\bf Proof.}}

\newcommand{\RR}{{\mathbb{R}}}
\newcommand{\ZZ}{{\mathbb{Z}}}

\begin{document}

\title{Stabilization of Heegaard splittings}
\author{Joel Hass}
\address{Department of Mathematics, University of California, Davis
California 95616}
\email{hass@math.ucdavis.edu}
\thanks{All three authors were partially supported by NSF grants}
\author{Abigail Thompson}
\address{Department of Mathematics, University of California, Davis
California 95616}
\email{thompson@math.ucdavis.edu}
 \author{William Thurston}
\address{Department of Mathematics, Cornell University,
Ithaca, NY 14853}
\email{wpt@math.cornell.edu}

\date{\today}

\begin{abstract}
For each $g \ge 2$ there is a 3-manifold with two genus $g$ Heegaard splittings  that require $g$ stabilizations to become equivalent.  
Previously known examples required at most one stabilization.  Control of families of Heegaard surfaces is obtained through a deformation to harmonic maps. 
\end{abstract}

\maketitle

\section{Introduction} \label{intro}

While area minimizing surfaces have proved to be a powerful tool in the study of 3-dimensional manifolds, harmonic maps of surfaces to 3-manifolds have not been as widely applied, due to several limitations. A homotopy class of surfaces generally gives rise to a large space of harmonic maps. A harmonic map of a surface need not minimize self-intersections and may fail to be be immersed.  In negatively-curved 3-manifolds there is a unique harmonic map in each conformal class of metrics on the domain, and smooth families of surfaces give rise to smooth families of harmonic maps \cite{EellsSampson},\cite{Hartman}. In this paper we study Heegaard splittings of 3-manifolds using families of harmonic surfaces.

A {\em genus $g$ Heegaard splitting} of a 3-manifold $M$ is a decomposition of $M$ into two genus $g$ handlebodies with a common boundary. It is described by an ordered triple $(H_1, H_2, S)$ where each of  $H_1, H_2$ is a handlebody and the two handlebodies intersect along their common boundary $S$, called a {\em Heegaard surface}. 
Two Heegaard splittings $(H_1, H_2, S)$  and $(H_1', H_2', S')$  of $M$ are equivalent if an ambient isotopy of $M$ carries $(H_1, H_2, S)$  to $(H_1', H_2', S')$.
Every 3-manifold has a Heegaard splitting \cite{Moise}, and Heegaard splittings form one of the basic structures used to analyze and understand 3-manifolds.

Corresponding to the Heegaard splitting is a family of surfaces that sweep out the manifold, starting with a core of one handlebody and ending at a core of the second. This family is geometric controlled by deforming it to a family of harmonic maps.  When the manifold is negatively curved, harmonic maps of genus $g$ surfaces have uniformly bounded area. In the manifolds we consider, the geometry forces small area surfaces to line up with small area cross sections of the manifold. As a result we obtain obstructions to the equivalence of distinct Heegaard splittings.

A {\em stabilization} of a genus $g$ Heegaard surface is a surface of genus $g+1$ 
obtained by adding a 1-handle whose core is parallel to the surface. Such a surface splits the manifold
into two genus $g+1$ handlebodies, and thus gives a new Heegaard splitting.
Two genus $g$ Heegaard splittings are {\em $k$-stably equivalent} if they become equivalent after $k$ stabilizations.
Any two Heegaard splittings become equivalent after a sequence of  stabilizations  \cite{SInger}.
An upper bound on the number of stabilizations needed
to make two splittings  equivalent is known in some
cases. If $G_p$ and $G_q$ are splittings of genus $p$ and $q$ with $p \le q$, and $M$ is non-Haken, then Rubinstein and Scharlemann obtained an upper bound of $5p + 8q -9$ for the genus of a common stabilization \cite{RubinsteinScharlemann:96} . For all previously known examples of manifolds with distinct splittings, the splittings become equivalent after a single stabilization of the larger genus Heegaard surface. 
The question of whether a single stabilization always suffices is sometimes called the {\em stabilization conjecture} \cite{Kirby}  (Problem 3.89), \cite{Scharlemann}, \cite{Schultens}, \cite{Sedgewick}. In Section~\ref{mainproof} we show that this conjecture does not hold. There are pairs of genus $g$ splittings of a 3-manifold that require $g$ stabilizations to become equivalent.

\begin{theorem} \label{mainthm}
For each $g > 1$ there is a 3-manifold $M_g$ with two genus $g$ Heegaard splittings that require $g$ stabilizations to become equivalent.
\end{theorem}

We outline the idea in Section~\ref{outline}. In Section~\ref{construction} we describe
the construction of the 3-manifolds $M_g$.  We derive isoperimetric inequalities used in the proof
in Section~\ref{sec:isoperimetric} and discuss deformations of surfaces to harmonic maps
in Section~\ref{harmonic}. In Section~\ref{sweepouts} we show how a Heegaard splitting gives rise to a  family of Heegaard surfaces
that sweep out the 3-manifold.
The proof of Theorem~\ref{mainthm} is given in Section~\ref{mainproof}. Finally in Section~\ref{I-bundle}we show how a somewhat weaker result can be obtained for an easily constructed class of hyperbolic manifolds.

This result was presented at the American Institute of Mathematics Conference on {\em Triangulations, Heegaard Splittings, and Hyperbolic Geometry} held in December 2007. At this conference D. Bachman announced, using different methods, examples giving a lower bound of $g-4$ for the number of required stabilizations.

\section{Outline of the argument} \label{outline}

Let $M_\phi$ be a hyperbolic 3-manifold that fibers over $S^1$ with monodromy $\phi$ and let
 $\tilde M_\phi$ denote its infinite cyclic cover.
A pictorial representation of $M_\phi$  is given in Figure~\ref{fig:Mphi}. Cutting open $M_\phi$ along a fiber gives a fundamental domain $B$ of the $\ZZ$-action on the infinite cyclic cover, which we call a {\em block}.
Blocks are homeomorphic, but not isometric, to the product of a surface and an interval. They are foliated by fibers of $M_\phi$, which in $B$ we call {\em slices}.
By cutting open a cyclic cover of $M_\phi$ we obtain a hyperbolic 3-manifold with as many adjacent blocks as we wish.

\begin{figure}[htbp]  
\centering
\includegraphics[width=4in]{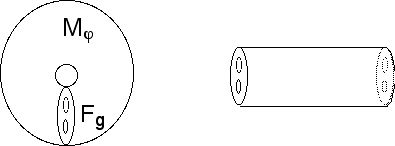} 
\caption{ $M_\phi$ and one block $B$, a fundamental domain of the $\ZZ$-action on its infinite cyclic cover. $B$  is homeomorphic to $F_g \times I$.}
\label{fig:Mphi}
\end{figure}

The manifold  $M_g$ used in our main result has pinched negative curvature. It contains two handlebodies $H_L$ and $H_R$ with fixed Riemannian metrics,  separated by a region homeomorphic to the product of a surface of genus $g$ with an interval. This intermediate piece is hyperbolic and isometric to $2n$ adjacent blocks.  The first $n$ blocks form a submanifold called $L$ and the next $n$ form a submanifold called $R$, as in Figure~\ref{fig:Mg}.  The value of $n$ can be chosen as large as desired without changing the geometry of $H_L$ and $H_R$. Details are in Section~\ref{construction}.

\begin{figure}[htbp] 
\centering
\includegraphics[width=4in]{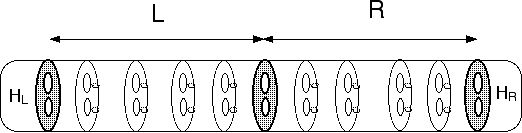} 
\caption{$M_2$, constructed with 10 blocks, 5 in each of $L$ and $R$.}
\label{fig:Mg}
\end{figure}

$M_g$ has two obvious Heegaard splittings $E_0 = (H_L \cup L, H_R \cup R, S)$ and $E_1 = (H_R \cup R,H_L \cup L, -S)$, where $S$ is a surface of genus $g$ separating $H_L \cup L$ and $H_R \cup R$ and $-S$ indicates $S$ with reversed orientation. We will show that these splittings are not  $k$-stably equivalent for $k <g$.

Let  $G_0$ be the Heegaard splitting obtained by
stabilizing $E_0$ $(g-1)$ times and $G_1 $ be the Heegaard splitting obtained by stabilizing  $E_1$ $(g-1)$ times. 
If $G_0$ and $G_1$ are equivalent, then there is
an isotopy $\{ I_s, ~ 0 \le   s  \le 1\}$ of $M_g$ that induces a diffeomorphism of $M_g$ carrying $G_0$ to  $G_1$,
A family of Heegaard splittings $\{ G_s = I_s(G_0),  ~ 0 \le s \le 1\} $ interpolates between 
$G_0$ and  $G_1$.

Associated to each Heegaard splitting $G_s$ is a  family of surfaces $F_{s,t}$ sweeping out $M_g$ from one core to the other.
In Section~\ref{harmonic} we show that such a family of surfaces can be deformed to a family of harmonic, or energy minimizing, maps.  These harmonic surfaces have area uniformly bounded by a constant $A_0$.  
This area bound restricts the way that surfaces can divide up the volume of $M_g$.  The surface cannot simultaneously split in half the volumes of $L$ and $L\cup R$.
In Section~\ref{mainproof} we use this to show that surfaces in such bounded area families cannot interpolate between $G_0$ and  $G_1$.  We conclude that $E_0$ and $E_1$ are not $k$-stably equivalent for $k <g$.

\noindent
{\bf Remark.}  The two Heegaard splittings $E_0$ and $E_1$ of $M_g$ do become equivalent after $g$ stabilizations. After adding $g$ trivial handles, a Heegaard surface can be isotoped to a surface formed by taking two genus $g$ slices connected by a thin tube. 
This can be isotoped so that one slice splits the volume of $L$ in half while the
second does the same to $R$, and
together they bisect the volume of $L \cup R$.
This type of surface arises in a family of bounded area surfaces
interpolating between genus $2g$ stabilizations
of $E_0$ and  $E_1$.

\section{Construction of $M_g$} \label{construction}

In this section we describe a certain Riemannian 3-manifold $M_g$ of Heegaard genus $g$,
based on the work of Namazi and Souto \cite{NamaziSouto}.  In Section~\ref{I-bundle} we give a simpler construction of a hyperbolic manifold that gives slightly weaker lower bounds on the number of stabilizations needed to make two Heegaard splittings equivalent.

We begin with a hyperbolic manifold $M_\phi$
that fibers over the circle with fiber a genus $g$ surface. 
Fix a fibration of $M_\phi$ with fibers $\{ S_t, ~ 0 \le t \le 1\}$ satisfying $\phi(S_0) = S_1$.
Define a {\em block}  $B$ to be the manifold obtained by cutting open $M_\phi$ along $ S_0$ and a {\em block manifold}  $B_n$ to be a union of $n$ blocks,
formed by placing end to end $n$ copies of $B$. 
The manifold $B_n$ is topologically, though not geometrically, the product of a
surface of genus $g$ and an interval. Its geometry can be obtained by cutting open the
$n$-fold cover of $M_\phi$ along a lift of $ S_0$. 
We call the fibers of  $B_n$ {\em slices} and label them by $S_t, ~ 0 \le t \le n$. We label the union of all fibers between $S_{t_1}$ and $S_{t_2}$ by $B[t_1,t_2]$.

The properties we desire for the manifold  $M_g$ are that it is a union of four pieces as in Figure~\ref{fig:Mg}. Two pieces are genus $g$ handlebodies, $H_L$ and $H_R$, and the other two $L$  and $R$ are each homeomorphic, though not isometric, to a product $F_g \times [-1,1]$, where $F_g$ is a surface of genus $g$.  $L$  and $R$ are each hyperbolic, and isometric to  a block manifold $B_n$. The sectional curvatures of $M_g$ are pinched  between  $-1 - \epsilon_0$ and $-1 + \epsilon_0$,
where $\epsilon_0  >  0$ can be chosen arbitrarily small. For our purpose we take $\epsilon_0 =1/2$.  

Namazi-Souto produced manifolds very similar to $M_g$ \cite{NamaziSouto}. While the manifolds they construct are suitable for the constructions we give, the argument is simpler if we modify the metric slightly so that the middle partt of $M_g$ is precisely, rather than approximately hyperbolic. We require the metric on $M_g$ to be isometric to a union of blocks on $L \cup R$.  This can be arranged by perturbing the manifold $M$ in Theorem~5.1 of \cite{NamaziSouto}, which has sectional curvatures  pinched  between  $-1 - \epsilon$ and $-1 + \epsilon$ where $\epsilon >0$ can be chosen arbitrarily small.  This manifold has a product region separating two handlbodies with geometry $\epsilon$ close (in the $C^2$-metric) to the geometry of a block. A small $C^2$-perturbation gives a new metric on $M$ with sectional curvatures between  $-3/2$ and $-1/2$ everywhere and exactly isometric to a block on an interval separating two handlebodies.  Now cut open the manifold along a slice in this hyperbolic block and insert a copy of $B_n$ to obtain $M_g$, with $n$ as large as desired. 
However large the choice of $n$, the geometry of $M_g$ is unchanged on the two complementary handlebodies of $B_n$, which we call $H_L$ and $H_R $.

\section{Isoperimetric inequalities} \label{sec:isoperimetric}
In this section we develop some isoperimetric inequalities for curves in surfaces and for surfaces in 3-manifolds.  
The following isoperimetric  inequality holds  for a curve in hyperbolic space $H^2$ \cite{Chavel}. A proof can be obtained by using symmetry to show that a round circle bounds at least as much area as any other curve, using no more length.  Explicit formulas for length and area in hyperbolic space then imply the inequality. The curve does not need to be embedded or connected.

\begin{lemma}  \label{H2curve} 
Let $c$ be a closed curve in $H^2$ with  length $L$. Then $c$ is the boundary of a disk $f: D^2 \to H^2$ of area $A$, with $A < L$.
\end{lemma}

We now consider more general isoperimetric inequalities for curves in surfaces and surfaces in 3-manifolds.  In these settings we
consider the areas and volumes of 2-chains and 3-chains with boundary.  
We first give an isoperimetric inequality for curves in a compact family of Riemannian metrics on a surface.

\begin{lemma} \label{surfaceisoperimetric}
Let $(S,g_r)$ be a closed 2-dimensional surface with Riemannian metric $g_r, r \in U$, where $g_r$ is a family of metrics smoothly parameterized by a compact set
$U \subset \RR^n$.
There is a constant $K$  such that for any surface $S$ in the family $(S,g_r)$,
\begin{enumerate}
\item A null-homotopic curve  $c$ in $S$ is the boundary of a disk $f: D \to S$
with 
$$
 \mbox{Area}(f(D))  \le K \cdot \mbox{Length}(c) .
$$
\item A collection of curves   $\{ c = \cup c_i \}$, with each $c_i$ null-homotopic in $S$, is the boundary of a collection of disks $\{ f_i : D_i \to S \}$
with 
$$
\sum  \mbox{Area}(f_i(D_i))  \le K \cdot \mbox{Length}(c) .
$$
\item A null-homologous curve $c$ on $S$
bounds a 2-chain $X_2$ in $S$ with
$$
 \mbox{Area}(X_2)  \le K \cdot \mbox{Length}(c) .
$$
\end{enumerate}
\end{lemma}
\bproof
First consider the case where $c$ is a null-homotopic curve. For a hyperbolic surface $S$, $c$ lifts to $H^2$ where it bounds a mapped in disk of area less than  
length($c$) by Lemma~\ref{H2curve}.
The surface $(S,g_r)$ is conformally equivalent to a hyperbolic surface,
so there is a hyperbolic metric $(S,g)$ and a diffeomorphism from
$(S,g)$ to $(S,g_r)$ that stretches or compresses any vector in the tangent space
of $(S,g)$ by a factor of at most $\lambda$. A single choice of $\lambda$ can be made for all the surfaces in the compact family $(S,g_r)$. Lengths of curves measured
in $(S,g)$ differ from lengths in $(S,g_r)$ by a factor of at most
$\lambda$ and the area of a region in $(S,g)$  differs from the area
in $(S,h_r)$ by a factor of at most $\lambda^2$. 
By comparing the area of a spanning disk 
and $\mbox{Length}(c)$ in  $(S,g)$, we see that in  $(S,g_r)$
$$
\mbox{Area}(f(D))  \le \lambda^3 \cdot \mbox{Length}(c) .
$$
The same argument applies for a  collection of null-homotopic components, 

There is a constant $\epsilon$ that gives a lower bound for the injectivity radius for the family of metrics $h_r$.  If the length of $c$ is less than $\epsilon$ than each of the components of $c$ is null-homotopic and the above bound holds for the area of a collection of spanning disks.

Now consider a general null-homologous curve $c$ of length greater than $\epsilon$. A small perturbation and  orientation preserving cut and paste transforms $c$ into a union of disjoint embedded components that is homologous to $c$. A homology between this embedded curve and $c$ has arbitrarily small area, so we can assume that  $c$ is a collection of embedded disjoint curves.
Each side of $c$ gives a 2-chain with boundary $c$, and so $ \mbox{Area}(S)$ gives an upper bound to the area of a 2-chain $X_2$ spanning $c$. Therefore
$$
\mbox{Area}(X_2)  \le   \mbox{Area}(S) \le (\mbox{Area}(S) / \epsilon ) \cdot \mbox{Length}(c) .
$$
The Lemma follows with $K= \lambda^3  + \mbox{Area}(S)/ \epsilon$.
\qed

To simplify the next calculations, it is convenient to work with chains having $\ZZ_2$-coefficients. With this choice we do not need to pay attention to orientations, signs, or multiplicity other than even or odd.  A 2-chain spanning a general position curve in a surface is obtained by  2-coloring the complement of the curve, dividing the surface into black and white regions separated by the curve. The area of a spanning 2-chain is the area of  either the white or the black subsurface, and  has a value between zero and the area of the surface. The two choices of  spanning 2-chain have area that sums to the area of the entire surface.  A similar statement applies for 3-chains spanning a surface in a 3-manifold.

We now analyze the geometry of surfaces in block manifolds.
The metrics on the slices $S_t$ of a block manifold, up to isometry, are parameterized by a closed interval, so Lemma~\ref{surfaceisoperimetric} applies. Therefore there is a constant $K_1$ independent of $n$, such that  a null-homologous curve $c$ on a slice $S_t$
bounds a 2-chain $X_2$ in that slice.with
$$
 \mbox{Area}(X_2)  \le K_1 \cdot \mbox{Length}(c) .
$$ 

Each slice is an incompressible surface in $B_n$ and therefore
with its induced metric has injectivity radius greater or equal to that of $M_\phi$.
Set $\epsilon$ to be the constant
$$
\epsilon = \min_{t \in [0,1]}  \{   \mbox{injectivity radius}(M_\phi) , \mbox{Area}(S_t) / 10 K_1 \}
$$

Now consider a proper separating surface $ (F, \partial F)$ in $(B_n, \partial B_n)$.
For each $t$, the complement of  $F \cap S_t $ in the slice $S_t$ can be 2-colored.
In slices $S_t$ where $\mbox{Length}(F \cap S_{t}) \le \epsilon$, each curve of
$F \cap S_{t}$ is null-homotopic.
Lemma~\ref{surfaceisoperimetric} implies that there is a collection of disks spanning
the components of $F \cap S_{t}$ whose total area is less than $  \mbox{Area}(S_t) / 10$.
Call the subsurface of $S_t$ missing these disks the {\em large side} of $F \cap S_t$. Alternately, this subsurface is the largest of  the connected complementary components of $F \cap S_t$ in $S_t$.

\begin{lemma}  \label{gbound}
Suppose $(F, \partial F)$ is a proper separating surface   in $(B_n, \partial B_n)$  and intersects slices $S_{t_1}, S_{t_2}$ transversely with  $ \mbox{Length}(F \cap S_{t_1}) \le \epsilon$ and  $ \mbox{Length}(F \cap S_{t_2}) \le \epsilon$.
If the genus of $F$ is less than $g$ then an arc $\alpha$ from  the large side of $F \cap S_{t_1}$ to the large side of $F \cap S_{t_2}$ intersects $F$ algebraically zero times.
\end{lemma}
\begin{proof}
Cap off the curves in $F \cap S_{t_1}, F \cap S_{t_2}$, each of which is shorter than the injectivity radius, by adding disks missing the large side of each surface. These disks do not intersect the arc $\alpha$. The resulting closed surface in the region between  $S_{t_1} $ and $S_{t_2}$, has genus less than $g$. Such a surface is null-homologous in the  region between these two slices, which is homeomorphic to the product of a genus $g$ surface and an interval, and therefore cannot algebraically separate them.
\end{proof}

The next lemma gives a 3-dimensional isoperimetric inequality for surfaces in block manifolds.
Note that Lemma~\ref{3disoperimetric} is false if the genus $k \ge g$.  
A  slice can of genus $g$ can split the volume of $B_n$ in half.
 
\begin{lemma} \label{3disoperimetric}
Let  $ (F, \partial F) \subset  (B_n, \partial B_n)$ a null-homologous proper surface in a block manifold,  having genus less than $g$ and area less than some constant $A_0$. 
Then there is a constant $B_0$, independent of $n$, such that  $F$ bounds
a 3-chain $X$ in $B_n $ with $\mbox{Volume}(X) < B_0$.
For $n$ sufficiently large,  
$$
\mbox{Volume}(X)  < (.1) \mbox{Volume}(B_n).
$$
\end{lemma}
\begin{proof}
Since $F$ is null-homologous, we can 2-color its complement.
This gives two 3-chains with boundary $F$, and we take $X$ to be the one of smaller volume.
The coarea formula \cite{Federer} implies that the volume of $X$ is comparable to the integral of its cross-sectional areas. Namely there is a constant $C_1$ such that
\begin{equation}
\mbox{Volume}(X)  / C_1 < \int_0^n \mbox{Area}(X \cap S_t) ~ dt < C_1 \cdot \mbox{Volume}(X).
 \label{eqn:coarea1}
\end{equation}
The coarea formula also gives an inequality between the area of $F$ in $B_n$ and the lengths of its intersections with slices. There is a constant $C_2$ such that
\begin{equation}
\int_0^n \mbox{Length}(F \cap S_t) ~ dt < C_2  \cdot \mbox{Area}(F).
 \label{eqn:coarea2}
\end{equation}
The constants $C_1$ and $C_2$ are determined by the geometry of a single block $B$, and do not depend on $n$.

We estimate the volume of $X$ by integrating the area of its cross sections.  To do so, we need to consistently specify for each slice $S_t$ 
which of the two complementary subsurfaces of $F \cap S_t$ lies in $X$.
If $ \mbox{Length}(F \cap S_t) < \epsilon$ then we take the side of $F$ containing the complementary subsurface of smaller area in $S_t$.  This subsurface has area less than $ K_1\epsilon$ by Lemma~\ref{surfaceisoperimetric} and misses the large side of $F \cap S_t$.  
Each such subsurface lies on the same side of $F$ (mod  2) by Lemma~\ref{gbound}.  For slices in which $ \mbox{Length}(F \cap S_t) \ge \epsilon$ the construction given in Lemma~\ref{surfaceisoperimetric} gives a bound 
$ K_1 \cdot \mbox{Length}(F \cap S_t)$ for the area of both complementary sides in $S_t$ of $ F \cap S_t$, and in particular for the side lying in $X$. By Equation~(\ref{eqn:coarea1}), $X$ has volume bounded above by 
\begin{eqnarray*}
  \mbox{Volume}(X) < C_1  \int_0^n  \mbox{Area}(X \cap S_t) ~ dt 
  < C_1 K_1  \int_0^n   \mbox{Length}(F \cap S_t)  ~ dt 
   \\ <  C_1 K_1 C_2 A_0 = B_0.
\end{eqnarray*}
This bound is independent of $n$, so for $n$ sufficiently large it is less than $\mbox{Volume}(B_n) /10$.
\end{proof}

Finally we  give a 3-dimensional isoperimetric inequality that implies  small area surfaces in a Riemannian manifold bound regions of small volume. The argument follows the line of an argument of Meeks \cite{Meeks}. The result is valid for both integer and $\ZZ_2$ coefficients, though we need only the latter.
   
\begin{lemma} \label{3dgeneral}
Let $M$ be a compact, connected Riemannian 3-manifold and $ F $ a proper surface in $M$, not necessarily embedded or connected.
Given $\epsilon_0>0 $ there is a $\delta_0 >0 $ such that if   $\mbox{Area}(F)  < \delta_0$ then
$F$ bounds a 3-chain $X$ in $M$ with $\mbox{Volume}(X)  <  \epsilon_0$.
\end{lemma}
\begin{proof}
By cut and paste we can reduce to the case where $F$ is embedded, though not necessarily connected.
The result is true for a surface in the unit ball in $\RR^3$, where we have the inequality
$$
V  \le    \frac{ (A)^{3/2} }{ 6 \sqrt{\pi} }.
$$
for the area of a surface $A$ and the volume  $V$ of the region it bounds.

For a fixed Riemannian metric on a ball, there is a bound $\lambda$ to the maximum stretch or compression of the length of a vector.  This implies a bound on the volume increase of  a region of $\lambda^3$ and a bound on the change in the area of a surface of $\lambda^2$. It follows that
\begin{equation}
V \le  C \cdot  A^{3/2}      \label{ballisop}
\end{equation}
for some constant $C$ depending only on the metric.

Now take a cell decomposition of $M$ with a single 3-cell $R$.  
By the coarea formula and the area bound for $F$, for  sufficiently small $\delta_0$ the 2-complex forming the boundary of this 3-cell $\partial R$ can be perturbed so that the length of its intersection
with $F$ is arbitrarily short.  The isoperimetric inequality for the boundary 2-sphere of $R$ allows us to
surger $F$ along short intersection curves with the boundary, while keeping its total area
less than a  constant $\delta_1$ satisfying $\lim_{\delta_0 \to 0} \delta_1 = 0$.  The surgered surface lies in a Riemannian 3-ball in which we can apply Equation~\ref{ballisop}.  The resulting small volume region in a ball gives a region $X$ in $M$ with $\partial X = F$, since
$X$  has canceling boundary (mod 2) on $\partial R$.  Furthermore the volume of $X$ approaches 0 as  $\delta_0 \to 0.$  In particular we can arrange that   $\mbox{Volume}(X)  <  \epsilon_0$ by choosing $\delta_0$ sufficiently small.
\end{proof}

\section{Harmonic and bounded area maps} \label{harmonic}

Eells and Sampson showed that a map from a Riemannian surface $F$ with metric $h$
to a negatively curved manifold $M$ is homotopic to a harmonic map
\cite{EellsSampson}. This harmonic map is unique unless its image is a point or a closed geodesic \cite{Hartman}, cases that we will not need to consider.  
It is obtained by deforming an initial map of a surface with a given Riemannian metric (or conformal structure) in the direction of fastest decrease for the energy, called the tension field .
The resulting harmonic map depends smoothly on the metrics of both the domain and the image, as shown in Theorem~3.1 of \cite{EellsLemaire}. See also \cite{Sampson}.  A homotopy class of maps from a surface to a 3-manifold is {\em elementary}  if it induces a homomorphism of fundamental groups whose image is trivial or cyclic.

\begin{lemma} \label{harmonicfamily}
Let $\{ g_u: u \in U \}$ be a family of metrics on a closed surface, parametrized by some compact set $U \subset \RR^n$ and let $M$ be a closed manifold with negative sectional curvature.
Then a smooth family of non-elementary maps $\{ f_u: (F, g_u)  \to M \}$ is
deformable to a smooth family of harmonic maps $\{ h_u: (F, g_u)  \to M \}$.
\end{lemma}

Since harmonic maps satisfy a maximal principle, they are more negatively curved then the ambient manifold wherever they are immersed \cite{Sampson}. The Gauss-Bonnet theorem then implies an area bound proportional to the genus. The immersion assumption can be removed by an approximation of a general map by immersions. A detailed argument is found in Theorem~3.2 of \cite{Minsky:1992}.

\begin{lemma} \label{harmonicarea}
A harmonic map $f:F \to M$ from a genus $g$ surface to a hyperbolic 3-manifold $M$  has area bounded above by $ 4\pi (g-1)$.  If $M$ has sectional curvatures between $-s$ and $-r$, for $0< r<s$, then its area is bounded above by $ 4\pi (g-1) /r$.
\end{lemma}

For any Riemannian metric on $F$, a basic inequality relates the energy and area of a map $f$ \cite{EellsLemaire}:
$$
\mbox{Energy}(f) \ge 2 \mbox{Area}(f).
$$

Equality occurs precisely when $f$ is almost conformal, i.e. conformal except possibly at finitely many singular points with zero derivative. In particular, equality holds for an isometric immersion, an immersion of a surface into a Riemannian manifold with the induced metric on the surface.

\begin{lemma} \label{areadrops}
A smooth family of immersions $\{ f_{u,0}: F  \to M, u \in U \}$ from a surface $F$ to a closed negatively curved 3-manifold $M$ is  smoothly deformable to a family of harmonic maps $\{ f_{u,1}: (F, g_u)  \to M\}$ through maps $\{ f_{u,s}, ~ 0 \le  s \le 1 \}$ satisfying
$$
\mbox{Area}(f_{u,s}) \le \mbox{Area}(f_{u,0}).
$$
\end{lemma} 
\begin{proof}
We begin by taking the induced metric on the surface $F$ using each of the family of maps $ f_{u,0}$. This gives a family of isometric immersions of $F$, for each of which the energy equals twice the area. The Eells-Sampson process of deforming along the tension field gives an energy decreasing deformation that converges to a harmonic map \cite{EellsSampson}. Since the energy is non-increasing during this flow, and the area is bounded above by twice the energy, the area of each surface in this flow is bounded above by its initial value. 
\end{proof}

\section{Sweepouts} \label{sweepouts}

Consider a family of maps $k: F \times (-1,1) \to M$ such that 
$$
\lim_{t \to \pm 1} \mbox{Area}(k (F, t)) =0
$$ 
and let $F_t$ denote  the surface $k (F, t)$.
By Lemma~\ref{3dgeneral} we know that the volume of a 3-chain bounded by $F_t$ also approaces zero as $t \to \pm 1$.
It follows that there is a $\beta>0$ such that one side of $F_t$ has volume less than  $\mbox{Volume}(M)/10$ if $t \in (-1, -1 + \beta] \cup [1-\beta, 1 )$. In particular $F_{-1 + \beta}$ bounds a 3-chain $ C_{-1 + \beta}$ of volume less than   $\mbox{Volume}(M)/10$ and similarly $F_{1 - \beta}$ bounds such a chain $C_{1 - \beta}$.
Consider the 3-cycle $Z \in H_3(M; \ZZ_2)$ formed  by adding the 3-chains $k(F \times [-1 + \beta, 1-\beta ])$, $ C_{-1 + \beta}$ and $C_{1 - \beta}$.  
We say the family of maps $k $ has {\em degree one} if $Z$ represents the fundamental homology class of $M$ (with $Z_2$-coefficients) and in that case we define it to be a {\em sweepout} of $M$.
Note that changing the value of $\beta$ continuously changes the volume of $Z$, while a change in the homology class of $Z$ would cause the volume to change by the volume of $M$.  Thus the choice of $\beta$ does not affect the question of whether $k $ has degree one.

We digress somewhat to point out that our constructions give a new approach to constructing sweepouts of bounded area.
Pitts and Rubinstein showed the existence of an unstable minimal surface in a 3-manifold using a minimax argument.
The minimal surface they construct has maximal area in a 1-parameter family of surfaces obtained from a sweepout given by a strongly irreducible Heegaard splitting.  Since a genus $g$ minimal surface in a hyperbolic manifold has area less than $4\pi (g-1)$, this implies an area bound for each surface in the sweepout.  
The existence of a sweepout composed of bounded area surfaces has implications on the geometry and Heegaard genusof a 3-manifold, as noted in Rubinstein \cite{Rubinstein} and Bachman-Cooper-White \cite{BCW}.
Harmonic maps give an alternate way to obtain the same area bound implied by Pitts-Rubinstein, without finding a minimal surface and without assuming that a Heegaard splitting is strongly irreducible.

\begin{theorem} \label{boundedarea}
If $M$ is a hyperbolic 3-manifold with a genus $g$ Heegaard splitting then $M$ has a sweepout whose
surfaces each has area bounded above by  $ 4\pi (g-1)$.
\end{theorem}
\begin{proof}
Parametrize each surface of the Heegaard splitting to obtain a smoothly varying family of embedded maps $\{ f_t: (F, g_t)  \to M ,~ -1 < t < 1\}$, where $g_t$ is the induced metric on $F$ under the map $f_t$. By Lemma~\ref{harmonicfamily}, the family of maps is
deformable to a smooth sweepout of harmonic maps $\{ h_t: (F, g_t)  \to M \}$. Each map $h_t$ has area bounded above by  $ 4\pi (g-1)$, by Lemma~\ref{harmonicarea}.
\end{proof}

Given a sweepout of a 3-manifold $M$ and two subsets $L$ and $R$,
we can characterize the direction of  the sweepout relative to $L$ and $R$, describing which of $L$ and $R$ is first engulfed.
Let $\beta$ be a constant that satisfies the conditions of Lemma~\ref{3dgeneral}, so that for $-1 < s \le -1+ \beta$,  $F_s$ bounds a 3-chain of volume less than $\mbox{Volume}(M)/10$. 
If $-1 < t < -1 +\beta$ define  $K_t$ to be the 3-chain $C_{t}$ with boundary $F_t$, given by the smaller volume side of $F_t$.
If $t \ge -1 + \beta$ define  $K_t$ to be the 3-chain obtained by adding the 3-chains  $k (F \times [-1+\beta, t]) $ and $ C_{-1 + \beta}$, again with boundary  $F_t$.
As $t$ increases on $(-1, 1)$ the volume of $K_t$ changes continuously, and satisfies 
$$
\lim_{t \to-1} \mbox{Volume}(K_t \cap (L \cup R) ) = 0
$$
and 
$$
\lim_{t \to 1} \mbox{Volume}(K_t  \cap (L \cup R) ) =   \mbox{Volume}(L \cup R).
$$
Let $Y$ denote the set of points $t \in (-1,1)$ where $K_t $ contains half the volume of $ \mbox{Volume}(L \cup R)$,
$$
Y = \{ t :  \frac{\mbox{Volume}( K_t \cap (L \cup R) )}{{\mbox{Volume}}(L \cup R) } = \frac{1}{2} \}.
$$
 For a generic sweepout this set is finite and odd.

\noindent
{\bf Definition.} 
A point $t \in Y$ is an {\em $L$ point} if 
$$
\mbox{Volume}( K_t \cap L)  >   \mbox{Volume}( K_t \cap R) 
 $$
and an {\em $R$ point} otherwise. 
A sweepout is an {\em $LR$-sweepout} if it has an odd number of $L$ points and an {\em $RL$-sweepout} if if it has an odd number of  $R$ points. 

The Heegaard splitting $E_0$ gives rise to an $LR$-sweepout that begins with surfaces near a graph at the core of $H_L$, sweeps out $L \cup R$ with embedded slices, and ends with surfaces that collapse to a graph at the core of $H_R $. The stabilized Heegaard splitting $G_0$ gives rise to an
$LR$-sweepout by embedded surfaces of genus $2g-1$.  A stabilization adds a loop to each of the graphs forming the cores of $E_0$, the two loops linking once in $M_g$, and each crossing the separating surface $S$.
Between the resulting graphs is a product region which can be filled with Heegaard surfaces of the stabilized splitting. See Figure~\ref{cores}, which shows two cores and a surface in $G_0$ that has been stabilized once. Assuming that $G_0$ and $G_1$ are equivalent, composing the resulting sweepout of $G_0$ with the diffeomorphisms $\{ I_s, ~ 0 \le s \le 1 \}$ gives a family of genus $2g-1$ sweepouts connecting $G_0$ and $G_1$.

\begin{figure}[htbp]  
\centering
\includegraphics[width=4in]{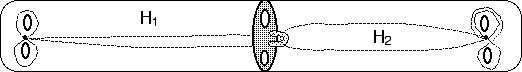} 
\caption{ Stabilized cores.}
\label{cores}
\end{figure}

\begin{lemma} \label{family}
Suppose that $E_0$ is equivalent to $E_1$ after $(g-1)$ stabilizations. 
Then there is a constant $A_0$, independent of the number of blocks $n$, and
a family of surfaces 
$$
\{  f_{s,t} : F  \to M_g, ~ 0 \le s \le 1, ~ -1 < t < 1  \}
$$
satisfying the following conditions:
\begin{enumerate}
%1
\item $F$ has genus $2g-1$
%2
\item  For each $s$, $ \{ f_{s,t} : F  \to M_g, ~ -1 < t < 1 \}$ is a sweepout.
%3
\item  $\{ f_{0,t}(F), ~ -1 < t < 1 \}$ is an $LR$-sweepout.
%4
\item   $\{ f_{1,t}(F), ~ -1 < t < 1  \}$ is an $RL$-sweepout.
%5
\item  Each surface in the two sweepouts $f_{0,t}(F)$ and $f_{1,t}(F)$ has area bounded above by $A_0$.   
%6
\item   Each surface $ \{ f_{s,t}(F), ~  0 \le  s   \le 1, ~ -1 < t  < 1 \}$ is embedded.
 \end{enumerate}
\end{lemma}
\begin{proof}
$M_g$ is formed from the union of a handlebody $H_L$, $n$-blocks forming $L$, $n$-bocks forming $R$, and a second handlebody $H_R $. The geometry of a single block, and of each of $H_L, H_R $, does not depend on $n$.  Pick a core for each of $H_L, H_R $ and foliate the complement of this core in each handlebody by embedded Heegaard surfaces, connecting the core to the slice that forms the boundary of $H_L$.  Do the same for $H_R $, and fill $L$ and $R$ with interpolating slices. This gives a foliation of the complement of the two cores in $M_g$  by genus $g$ leaves $\{L_t,~ ~ -1 < t < 1  \}$. 

Now stabilize by adding $(g-1)$ loops to each core graph and $(g-1)$ handles to each surface between the two core graphs.  By adding thin handles, this can be achieved while increasing the area of any surface by  less than $a_0$ and with the additional $(g-1)$ 1-handles bounding a region of volume less than $v_0$, where the constants $a_0$ and  $v_0$ can be chosen arbitrarily small, independently of $n$. For our purpose it suffices to pick $v_0 < V_B/10$, where $V_B$ is the volume of a single block, and $a_0 < 1$.

Let $F$ be a surface of genus $ 2g-1$ and
construct maps  $\{ f_{0,t} : F \to M_g, ~ -1 < t < 1  \}$ that smoothly parametrize the stabilized surfaces. 
We will refer to $f_{0,t}(F)$ as $F_{0,t}$.
Set the value of the constant $A_0$ to be the largest area of the surfaces $\{ F_{0,t}, ~ -1 < t < 1  \}$.
Note that our construction gives a value for $A_0$ that is independent of the number of blocks $n$ in $M_g$.

The Heegaard splitting $G_0 = (\bar H_1, \bar  H_2,  \bar S)$ obtained by stabilizing  $E_0$ $(g-1)$ times is assumed to be equivalent to the reversed splitting $G_1 = (\bar H_2, \bar  H_1, - \bar S)$, where  $\bar S = F_{0,0}$ and $- \bar S$ indicates the orientation of $\bar S$ has reversed. So an isotopy  $I_s , ~ 0 \le s \le 1 $ from the identity map $I_0$ to a diffeomorphism $I_1$ carries $(\bar H_1, \bar H_2, \bar S)$ to $(\bar H_2, \bar  H_1, - \bar S)$. Construct the family of surfaces $f_{s,t} : F  \to M_g$ by defining 
$$
f_{s,t} = I_s \circ f_{0,t} : F  \to M_g
$$
and let $F_{s,t}$ denote $f_{s,t} (F)$.

For any constant $\alpha > 0$ we can assume, by stretching out a collar around the invariant surface $F_{0,0}$, that  $I_1$ carries $F_{0,t}$ to  $F_{0,-t}$ for each  $t \in  [-1+ \alpha, 1- \alpha ] $.
For $\alpha$ sufficiently small the embedded surfaces 
$$
\{ F_{s,t }; t \in ( 0,-1+ \alpha ] \cup [ 1 - \alpha ,1 ) \}
$$
lie in small neighborhoods of the images of the  core graphs of  $G_0$ under $I_s$, have area uniformly bounded above by $a_0$, and bound submanifolds having volume less than $v_0$.

The Heegaard splitting $E_0$ gives rise to a sweepout of $M_g$  by genus $g$ surfaces $\{ L_t , ~ -1 < t < 1 \}$ starting near the core of $H_L$ and ending near the core of $H_R$. This sweepout foliates the complement of the two core graphs. The surface $L_t $ bounds a 3-chain $K_t$ that fills up the side containing $\{ \cup L_{t'} : t' <t \}$. $K_0$ bisects the volume of $L \cup R$  and contains $L$ but not $R$. The surfaces  $L_t $, when parametrized, give an $LR$-sweepout with  $\mbox{Volume}(K_0 \cap L) = nV_B$ while $\mbox{Volume}(K_0 \cap R) = 0$.  The surface $F_{0,t}$, obtained by stabilizing $L_t$, bounds a 3-chain $K_{0,t}$ whose volume of intersection with $L$ and $R$ differs by less than $v_0 < v_B/10$ from that of $K_t$. Therefore $F_{0,t}$ also gives an $LR$-sweepout.  By the same argument applied to the stabilization of $E_1$, we have that  $F_{1,t}$ gives an $RL$-sweepout. 

All properties now follow.
\end{proof}

We now show that a path of sweepouts in $M_g$ whose surfaces have area uniformly bounded by a constant $A_0$ cannot start with an $LR$-sweepout and end with $RL$-sweepout, if $n$ is sufficiently large.
Define 
$$
V_L =  \mbox{Volume } (L) =  \mbox{Volume } (R)
$$
sot that 
$
2V_L =  \mbox{Volume }(L \cup R)  .
$
 
\begin{lemma} \label{nofamily}
Suppose $A_0$ is a constant and that $M_g$ is constructed with $2n$-blocks.  
Then for $n$ sufficiently large 
there does not exist a smooth family of maps 
$$
\{ h_{s,t} : F  \to M_g, ~ 0 \le s \le 1, ~ -1 < t < 1  \} 
$$
from a  surface $F$  to $M_g$ satisfying the following conditions:
\begin{enumerate}
\item $F$ has genus less than $2g.$
\item   For each fixed $s$, $h_{s,t} : F  \to M_g$ is a sweepout.
\item  Each surface $h_{s,t}(F)$ has area bounded above by $A_0$.   
\item   $\{ h_{0,t}(F), ~ -1 < t < 1  \} $ is an $LR$-sweepout.
\item   $\{  h_{1,t}(F), ~ -1 < t < 1 \}  $ is an $RL$-sweepout.
\end{enumerate}
\end{lemma}
\begin{proof}
Suppose such a family exists and let $F_{s,t}  $ denote $ h_{s,t}(F)$.
For each $F_{s,t}$ we construct a 3-chain $K_{s,t}$ with boundary $F_{s,t}$.
Pick $\beta$ sufficiently small to satisfy the conditions of 
in Lemma~\ref{3dgeneral} for each $0 \le s \le 1$. 
Then for each $s$ and each $-1< t_0 \le -1+ \beta$ there is a 3-chain $ C_{s, t_0} $  with boundary  $F_{s, {t_0}} $ and volume less than $V_L/10$.
If $-1 < t < -1 +\beta$ take  $K_{s,t}$ to be equal to $ C_{s,t} $. 
If $-1 +\beta < t < -1 $ take  $K_{s,t}$ to be the sum of $ C_{s,-1+\beta} $ and 
$ \bigcup_{-1+ \beta  \le t' \le t} F_{s,t'} $.

The volume of $K_{s,t}$ varies continuously with $s$ and $t$, giving a continuous function from
$[0,1] \times (-1,1) \to \RR $. For each fixed $s$ the surfaces $F_{s,t}$ sweep out $M_g$ with degree one , so
$$
\lim_{t \to -1} \mbox{Volume} (K_{s,t} \cap (L \cup R) )=0
$$
and 
$$
\lim_{t \to 1} \mbox{Volume} (K_{s,t} \cap (L \cup R) )= 2V_L.
$$
Set
$$
Q = \{   (s,t) :  \frac {\mbox{Volume}(  K_{s,t} \cap (L \cup R)  )}{ 2V_L} = \frac{1}{2}  \} . 
$$
If necessary, perturb the value 1/2 used to to define $Q$ to a regular value of the volume function evaluated on the rectangle. We then have that  $Q$ is a 1-manifold and that
any path from the edge $t=-1$ to the edge $t=1$ of the $(s,t)$ rectangle must cross $Q$. 
It follows that there is a path contained in $Q$  connecting the edges $s=0$  and $s=1$ of the $(s,t)$ rectangle, as in Figure~\ref{st}.

\begin{figure}[htbp]  
\centering
\includegraphics[width=1.5in]{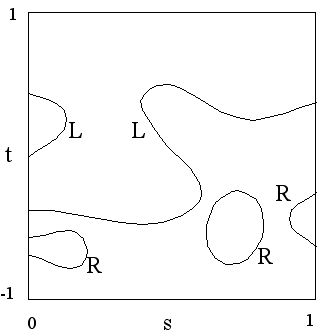} 
\caption{Each curve in $Q$ consists entirely of $L$ points or $R$ points.}
\label{st}
\end{figure}

We now claim that each component of $Q$ consists entirely of $L$ points or entirely of $R$ points.
If a component of $Q$ has both $L$ points and $R$ points then there is a point $ (s,t) \in Q $ on that component where 
$$
\mbox{Volume}(K_{s,t}  \cap L)   = \mbox{Volume}( K_{s,t} \cap R) 
$$
and since these two volumes sum to $V_L  $, we have
\begin{equation}
 \frac { \mbox{Volume}(  K_{s,t} \cap L ) }{ 2V_L} =  \frac { \mbox{Volume}(  K_{s,t} \cap R )  }{ 2V_L} =   \frac{1}{4} 
  \label{eqn:1/4}
\end{equation}
 
The area of  $F_{s,t}$ is bounded above by $A_0$, so Lemma~\ref{3disoperimetric} implies that for $n$ sufficiently large,   $F_{s,t}$ bounds a 3-chain in $L$ with volume less than $(0.1)V_L$.
Since any two 3-chains in $L$ with the same boundary have volumes in $L$ that sum to a multiple of $V_L $, we have 
$$
0  \le \frac{ \mbox{Volume}( K_{s,t} \cap L) } { V_L }  \le 0.1 .
$$
or
$$
0.9 \le \frac{ \mbox{Volume}( K_{s,t} \cap L) } { V_L }  \le 1 .
$$
As a fraction of $2V_L$
$$
0  \le \frac { \mbox{Volume}(K_{s,t} \cap L ) } {2V_L}   \le  0.05 
$$ 
or
$$
0.45  \le   \frac { \mbox{Volume}(K_{s,t} \cap L ) } {2V_L}  \le   0.5 .
$$ 
This contradicts Equation~\ref{eqn:1/4}. We conclude that components of $Q$ consist either entirely of  $L$ or entirely of $R$ points.

A component curve in $Q$ either meets each edge
$s=0$ and $s=1$ once, meets one edge of the $(s,t)$ rectangle twice, or is disjoint from both.
The parity of the number of $L$ points and the number of $R$ points on the two families
$h_{0,t}(F)$ and  $h_{1,t}(F)$ is the same for each edge $s=0$ and $s=1$, and thus either both are $LR$-sweepouts or both are $RL$-sweepouts.  
\end{proof}

We now show how to replace a family of sweepouts that starts and ends with sweepouts whose surfaces have area less than $A_0$, but with no control on the area of surfaces in the intermediate sweeputs, with a homotopic family of sweepouts whose surfaces have area uniformly bounded by $A_0$. 

\begin{lemma} \label{harmonize}
Let $\{ f_{s,t,0}: F \to M_g, ~ 0 \le s \le 1,~ -1 <   t  <  1\}$ be a family of surface maps such that
\begin{enumerate}
%1
\item  The genus of $F$ is less than $2g$.
%2
\item  For each $s$, $ \{ f_{s,t,0} : F  \to M_g , ~ -1 <  t <  1\}$ is a sweepout.
%3
\item  Each surface in the two sweepouts $f_{0,t,0}(F)$ and $f_{1,t,0}(F)$ has area bounded above by a constant $A_0$.   
%4
\item  $f_{0,t,0}(F), ~ -1 < t < 1  $ is an $LR$ sweepout.
%5
\item  $f_{1,t,0}(F), ~ -1 < t < 1  $ is an $RL$ sweepout.
%6
\item Each surface $ \{ f_{s,t,0} : F  \to M_g , ~ 0 \le s \le 1, ~ -1  <  t < 1\}$ is immersed.
\end{enumerate}
Then if $n$ is sufficiently large, there is a family of maps  
$$
\{ f_{s,t,u}: F \to M_g, ~ 0 \le s \le 1,~ -1\le  t \le 1, ~ 0 \le u \le 1\}
$$ 
such that
\begin{enumerate}
%1
\item  For each fixed $s,u$; $ \{ f_{s,t,u} : F  \to M_g, ~ -1\le  t \le 1 \}$ is a sweepout.
 %2
\item  Each surface $f_{s,t,1}(F)$ is harmonic, with area at most  $8\pi(2g-1)$. 
%3
\item   $f_{0,t,1}(F), ~ -1 < t < 1  $ is an $LR$ sweepout.
%4
\item  $f_{1,t,1}(F), ~ -1 < t < 1  $ is an $RL$ sweepout.
\end{enumerate}
\end{lemma}
\begin{proof}
We define a family of Riemannian metrics $h_{s,t}$ on $F$ for each  $~ 0 \le s \le 1,~ -1 <  t < 1$ by taking the induced metric pulled back from $M_g$ by the immersion $f_{s,t,0}$.  Then each map $f_{s,t,0}$ is conformal and has energy is equal to twice its area.  We  flow the family $f_{s,t,0}$ to a family of harmonic maps $f_{s,t,1}$ as in Lemma~\ref{areadrops}, with the area of each harmonic surface $f_{s,t,1}(F)$ less than $A_0$.
Each surface in the sweepout $f_{0,t,0}(F)$ has area bounded above by $A_0$, and this area bound is maintained for each surface $ \{ f_{0,t,u}(F), ~ 0 \le u \le 1  \}$ in the homotopy to harmonic maps, as in Lemma~\ref{harmonicarea}.  Similarly the homotopy to harmonic maps of the sweepout $f_{1,t,0}(F)$ given by  $ \{ f_{1,t,u}(F) , ~ 0 \le u \le 1  \}$   has surfaces whose area is uniformly bounded above by $A_0$.
Thus if $n$ is sufficiently large, the sweepout   $ \{ f_{0,t,1}(F) \}$  is an $LR$ sweepout, by Lemma~\ref{nofamily} and similarly $ \{ f_{1,t,1}(F)  \}$  is an $RL$ sweepout.
\end{proof}
 
 \section{Inverting Heegaard surfaces} \label{mainproof}

We now prove the main result:
\begin{proof}[Proof of Theorem~\ref{mainthm}]
We begin with the two genus $g$ Heegaard foliations, $E_0$ and $E_1$ of $M_g$.
We will show that these splittings are not $k$-stably equivalent for $k <g$ if $n$ is sufficiently large. It suffices to take $k=g-1$.

Assume to the contrary that for all $n$, $G_0$ and $G_1$ are  equivalent after $g-1$ stabilizations.
For $n$ large, Lemma~\ref{family} implies there is a family of sweepouts 
$$
\{ f_{s,t,0} : F  \to M_g, ~ 0 \le s \le 1, ~ -1 < t < 1  \}
$$ 
of genus $2g-1$ with $\{ f_{0,t,0}(F), ~ -1 < t < 1  \}$ an $LR$-sweepout,  $\{f_{1,t,0}(F), ~ -1 < t < 1  \}$ an $RL$-sweepout and that each surface in the two sweepouts $f_{0,t,0}(F)$ and $f_{1,t,0}(F)$ has area bounded above by $A_0$. 
Lemma~\ref{harmonize}  implies that this family    can be deformed to a new family of sweepouts 
$$ 
\{f_{s,t,1} : F  \to M_g, ~ 0 \le s \le 1, ~ -1 < t < 1  \}
$$ 
in which   $\{ f_{0,t,1}(F), ~ -1 < t < 1  \}$
is an $LR$-sweepout,  $\{f_{1,t,1}(F), ~ -1 < t < 1  \}$ is an $RL$-sweepout and each surface  $f_{s,t,1}(F)$ has area bounded above by $A_0$.  This is a path of sweepouts all of whose surfaces have area less than $A_0$ that connects an $LR$-sweepout  to a $RL$-sweepout.  Lemma~\ref{nofamily} states that no such path of sweepouts can exist, contradicting the assumption that fewer than $g$ stabilizations can make $E_0$ and $E_1$ equivalent.
\end{proof}

\section{A Hyperbolic Example} \label{I-bundle}

The Riemannian manifolds $M_g$ used in our construction were negatively curved, but not hyperbolic. We now show how to construct a family of hyperbolic manifolds that give somewhat weaker lower bounds on the number of stabilizations required to make two genus-$g$ splittings equivalent. Note that block manifolds cannot be isometrically embedded into a hyperbolic 3-manifold in which their slices are Heegaard surfaces, because the fibers of a surface bundle lift to planes in the universal cover. However there is an isometric embedding of block manifolds into a hyperbolic manifold in which slices are separating incompressible surfaces. These surfaces become Heegaard surfaces after two stabilizations. 

Let $N_0$ be a hyperbolic 3-manifold that is a union of two $I$-bundles over a non-orientable surface, glued along their common genus $k$ boundary surface, where $k$ is an integer greater than one. Such hyperbolic manifolds are double covered by a hyperbolic surface bundle over $S^1$. Some explicit examples can be found in \cite{Reid}. The boundary surface of each $I$-bundle has a neighborhood isometric to a neighborhood of a fiber in its double cover. We can cut open along this fiber and insert a block manifold with arbitrarily many blocks to obtain a hyperbolic manifold $N$, still constructed as a union of two I-bundles.  We let  $S$ denote a surface separating the two I-bundles in the center of the block manifold.

Removing a neighborhood of an interval fiber from an $I$-bundle with a genus $k$ boundary surface results in a handlebody of genus $k+1$.  Thus $N$ has a  genus $g=k+2$ Heegaard surface $S'$ obtained by adding a 1-handle to each side of $S$, with the core of each 1-handle an interval in each I-bundle. The two orderings of these handlebodies give rise to two Heegaard splittings of $N$, and to corresponding sweepouts that fill the two I-bundles in opposite order. The arguments used on $M_g$ in Section~\ref{mainproof} now apply to show that $N$ has two genus $g$ splittings that require no less than $g-4$ stabilizations to become equivalent.

 \end{document}